\documentclass{amsart}

\usepackage[utf8]{inputenc}
\usepackage{color}
\usepackage{xcolor}
\usepackage{hyperref}
\usepackage{amssymb}
\usepackage{float}
\usepackage{tikz,tikz-cd}
\usepackage{forest}
\usepackage{nicefrac}
\usepackage{enumerate}
\usepackage{mathtools}
\usepackage[nameinlink, capitalize, noabbrev]{cleveref}
\usepackage{natbib}
\usetikzlibrary{shapes.geometric,decorations.markings}
\usetikzlibrary{decorations.pathreplacing}
\usetikzlibrary{fit}
\usetikzlibrary{automata}
\usetikzlibrary{positioning}
\usetikzlibrary{intersections}
\tikzset{mytext/.style={font=\small, text=black}}

\hypersetup{
    colorlinks=true,
    linkcolor=teal,
    citecolor=magenta,
    }

\newtheorem{Theorem}{Theorem}

\newtheorem{Conjecture}{Conjecture}

\newtheorem{proposition}{Proposition}[section]
\newtheorem{lemma}[proposition]{Lemma}
\newtheorem{corollary}[proposition]{Corollary}

\newtheorem{Question}[Conjecture]{Question}

\theoremstyle{definition}
\newtheorem{remark}[proposition]{Remark}

\newcommand{\vr}{\le_{\mathrm{vr}}}

\newcommand\restr[2]{{
		\left.\kern-\nulldelimiterspace 
		#1 
		\littletaller 
		\right|_{#2} 
}}

\numberwithin{equation}{section}

\title{Virtual retracts in groups acting on rooted trees}
\author{Jorge Fariña-Asategui and Jon Merladet Urig\"uen}
\address{Jorge Fariña-Asategui: Centre for Mathematical Sciences, Lund University, 223 62 Lund, Sweden -- Department of Mathematics, University of the Basque Country UPV/EHU, 48080 Bilbao, Spain}
\email{jorge.farina\_asategui@math.lu.se}
\address{Jon Merladet Urig\"uen: CGTA, School of Mathematical Sciences, University of Southampton, Highfield, Southampton, SO17~1BJ, United Kingdom}
\email{J.F.Merladet@soton.ac.uk}
\thanks{The first author is supported by the Spanish Government, grant PID2020-117281GB-I00, partly with FEDER funds. The first author also acknowledges support from the Walter Gyllenberg Foundation from the Royal Physiographic Society of Lund}

\begin{document}

\begin{abstract}

We study virtual retracts in groups acting on rooted trees. We show that finitely generated branch groups do not have the local retraction (LR) property. Furthermore, we specialize to iterated monodromy groups of post-critically finite quadratic complex polynomials and show that the (LR) property characterizes, among post-critically finite quadratic complex polynomials, those with a euclidean orbifold, i.e. the powering map and the Chebyshev polynomial. Lastly, we show that periodic quadratic complex polynomials provide new examples of pro-$2$ groups with complete finitely generated Hausdorff spectrum.

\end{abstract}

\maketitle

\section{introduction}
\label{section: introduction}

Let $G$ be a finitely generated group. The \textit{profinite topology} in $G$ is given by declaring all cosets of finite-index subgroups of $G$ to form a basis of open sets. If the trivial subgroup of $G$ is closed in the profinite topology, we say that $G$ is \textit{residually finite}. Similar notions are the notions of \textit{subgroup separable} (or \textit{locally extended residually finite}) and \textit{extended residually finite}, meaning that finitely generated subgroups  and respectively all subgroups of $G$, are closed in the profinite topology. The notion of being residually finite is quite general. In fact, it is equivalent to $G$ acting faithfully on a spherically homogeneous rooted tree. On the other side of the spectrum, being extended residually finite is a quite restrictive property, only known to hold for virtually polycyclic groups \cite{Malcev} and certain abelian and nilpotent groups \cite{ERF}.

Subgroup separability lies somewhere in between the general property of being residually finite and the restrictive property of being extended residually finite. A celebrated result of Hall shows that free groups are subgroup separable \cite{Hall}. This has been generalized to surface groups, limit groups and (partially) to right-angled Coxeter groups by Scott \cite{Scott}, Wilton \cite{Wilton} and Haglund \cite{Haglund}, respectively. These results came along a stronger property: the local retraction property (LR).

A subgroup $H\le G$ is said to be a \textit{retract} of $G$ if there exists a group homomorphism $f:G\to H$ which restricts to the identity in $H$. A subgroup $H\le G$ is a \textit{virtual retract} of $G$ if there exists a finite-index subgroup $K\le_f G$ such that $H$ is a retract of $K$.  We say that $G$ has the \textit{local retraction (LR)} property if every finitely generated subgroup of~$G$ is a virtual retract of $G$ and the \textit{(VRC)} property if every cyclic subgroup is a virtual retract of $G$. Properties (LR) and (VRC) imply subgroup separability and cyclic subgroup separability (i.e. cyclic subgroups are closed in the profinite topology) respectively. The study of retracts was formally initiated by Long and Reid in \cite{Long-Reid} and expanded in \cite{J-A} and \cite{Ashot}; see \cite{ Brid-Wilt, Haglund} for more connections between separability and retracts.

As residually finite groups can be identified with subgroups of $\mathrm{Aut}~T$, i.e. the automorphism group of a spherically homogeneous rooted tree $T$, in what follows we focus on subgroups of $\mathrm{Aut}~T$. An important class of subgroups of $\mathrm{Aut}~T$ are the classes of branch and weakly branch groups, introduced by Grigorchuk in \cite{Branch}. For any vertex $v\in T$, we define the \textit{rigid vertex stabilizer} $\mathrm{rist}_G(v)$ as the subgroup of $G$ consisting of the elements in $G$ which only move descendants of $v$. For each level $n\ge 1$, the \textit{rigid level stabilizer} $\mathrm{Rist}_G(n)$ is the normal subgroup given by the direct product of the rigid vertex stabilizers of the vertices at level $n$. A subgroup $G\le \mathrm{Aut}~T$ is \textit{level-transitive} if it acts transitively on every level of $T$. A level-transitive subgroup whose rigid level stabilizers are infinite (or of finite index in $G$) is called \textit{weakly branch} (respectively \textit{branch}).

Subgroup separability of branch groups has been actively considered in the last decades. Grigorchuk and Wilson proved in \cite{GW} that the first Grigorchuk group is subgroup separable, a result later extended to the Gupta-Sidki 3-group by Garrido in \cite{Gar16} and to branch groups satisfying the so-called subgroup induction property by Francoeur and Leemann in \cite{FL}. Recently, the extended residually finite property has been studied in branch and weakly branch groups by the first author, Leemann and Nagnibeda in \cite{JPHT}.

Interestingly, subgroup separability on branch groups has not been proved via virtual retracts, contrary to the classical results on free groups \cite{Hall} and limit groups \cite{Wilton}. In fact, to the best of our knowledge, there has not been any detailed study of virtual retracts in branch groups, although related notions are being considered for certain branch groups by Bodart,  Carvahlo and Nyberg-Brodda in a work in progress (personal communication). Thus, our first goal in the present paper is to close this gap and study virtual retracts in branch and weakly branch groups. As (LR) is preserved by subgroups, Minasyan's example \cite[Example~3.6 and Proposition~3.7]{Ashot} shows that the first Grigorchuk group is not (LR). This motivates the following natural question:

\begin{Question}
\label{Question: branch with LR}
    Does there exist a finitely-generated branch group with property $\mathrm{(LR)}$?
\end{Question}

The first goal of the present paper is to settle \cref{Question: branch with LR}. In fact, we answer it in the negative:

\begin{Theorem}
\label{Theorem: branch not LR}
    Let $G$ be a finitely-generated branch group. Then $G$ is not $\mathrm{(LR)}$.
\end{Theorem}

The proof of \cref{Theorem: branch not LR} involves several steps. First, we generalize the example of Minasyan in \cite[Example 3.6 and Proposition 3.7]{Ashot} to prove that branch groups with torsion cannot have (LR) in \cref{lemma: LR with torsion}. Secondly, we characterize virtual retracts among normal subgroups of weakly branch groups. Indeed, we show that normal subgroups are virtual retracts if and only if they are of finite index; see \cref{lemma: normal subgroups when are they retracts}. Lastly, we characterize the (VRC) property in just-infinite branch groups in \cref{lemma: just-infinite VRC}, where we show that (VRC) is equivalent to the group being torsion. This concludes the proof of \cref{Theorem: branch not LR}. The remainder of the paper is devoted to an application of our results on virtual retracts to the orbifold topology of complex polynomials.

The first examples of branch groups introduced by Grigorchuk in the 1980s are instances of self-similar groups \cite{GrigorchukBurnside, GrigorchukMilnor}. Self-similar groups were introduced by Nekrashevych in the realm of complex dynamics; see \cite{SelfSimilar} for a detailed exposition. Given a complex rational function $f\in \mathbb{C}(x)$, we say that $f$ is a \textit{Thurston map} or \textit{post-critically finite} if the post-critical set $P_f:=\bigcup_{n\ge 1}f^n(C_f)$ is finite, where $C_f$ is the set of critical points of $f$. The fundamental group $\pi_1(\widehat{\mathbb{C}}\setminus P_f,z)$ is known to act on the tree of preimages $T $ of a point $z\in \widehat{\mathbb{C}}\setminus P_f$, where $\widehat{\mathbb{C}}$ denotes the Riemann sphere. This action induces a subgroup $\mathrm{IMG}(f)\le \mathrm{Aut}~T$ called the \textit{iterated monodromy group of $f$}.

Iterated monodromy groups are a complete combinatorial invariant of branched coverings of the Riemann sphere. They were used by Bartholdi and Nekrashevych in \cite{Thurston} to solve the well-known Hubbard's twisted rabbit problem and study the obstructed cases to which Thurston's criterion does not apply \cite{Douady}. Furthermore, iterated monodromy groups play an important role in classical problems in group theory, providing important examples of groups with exotic growth \cite{Erschler1,Erschler2, GrigorchukMilnor} and amenability properties \cite{RW2,RW1}.

Iterated monodromy groups have further important applications to several problems in arithmetic dynamics, including prime density problems in iterates of rational functions and the proportion of their periodic points over finite fields; compare \cite{Bridy, ABC, JonesAMS, JonesComp, JonesLMS, Juul, Odoni1,Odoni2}. A major recent breakthrough of the first author and Radi in \cite{FPP} shows that a complex polynomial $f\in \mathbb{C}[x]$ exhibits very different behaviour in relation to the above problems depending on whether $f$ is hyperbolic or euclidean in the following sense.

For a post-critically finite rational function  $f\in \mathbb{C}(x)$, its \textit{Thurston orbifold} is $(\widehat{\mathbb{C}},\nu_f)$, where $\nu_f$ is a map assigning to each post-critical point $p\in P_f$ a weight $\nu_f(p)\ge 2$ depending on the local degree of each critical point mapping to $p$ and $\nu_f(z)=1$ for any $z\in \widehat{\mathbb{C}}\setminus P_f$; see \cite{Douady}. The \textit{Euler characteristic} $\chi$ of the orbifold $(\widehat{\mathbb{C}},\nu_f)$ is defined as
$$\chi(\nu_f):=2-\sum_{z\in P_f}\left(1-\frac{1}{\nu_f(z)}\right).$$
A rational function $f\in \mathbb{C}(x)$ is said to be \textit{euclidean} if $\chi(\nu_f)=0$ and \textit{hyperbolic} if $\chi(\nu_f)<0$. 

The Thurston orbifold $(\widehat{\mathbb{C}},\nu_f)$ captures the topology of the post-critical orbits of~$f$. On the other hand, virtual retracts capture to some extent the profinite topology of the iterated monodromy group $\mathrm{IMG}(f)$. As the action of the group $\mathrm{IMG}(f)$ is given by the post-critical dynamics of $f$, we expect virtual retracts to be related to the topology of the Thurston orbifold $(\widehat{\mathbb{C}},\nu_f)$. Since hyperbolic rational maps are expected to yield iterated monodromy groups with weakly branch closures, compare \cite[Conjecture 3]{Hdim} and \cite[Theorem~C]{JorgeAV}, we propose the following conjecture in view of \cref{Theorem: branch not LR}:

\begin{Conjecture}
\label{conjecture: LR IMG}
    Let $f\in \mathbb{C}[x]$ be a post-critically finite rational function. Then, the following are equivalent:
    \begin{enumerate}[\normalfont(i)]
    \item $\mathrm{IMG}(f)$ has $\mathrm{(LR)}$;
    \item $\mathrm{IMG}(f)$ is virtually abelian;
    \item $f$ has a euclidean orbifold.
\end{enumerate}
\end{Conjecture}

Complex rational functions with a euclidean orbifold were completely classified by Douady and Hubbard in \cite{Douady}. For such a rational function $f$, the group $\mathrm{IMG}(f)$ can be identified with a virtually abelian group of affine transformations of the complex plane, and thus it has (LR). 

We confirm \cref{conjecture: LR IMG} for quadratic complex polynomials. Note that complex polynomials with a euclidean orbifold are linearly conjugate (up to a sign) to either the powering map or the Chebyshev polynomial \cite{Douady}:

\begin{Theorem}
\label{Theorem: LR quadratic}
Let $f\in \mathbb{C}[x]$ be a post-critically finite quadratic polynomial. Then, the following are equivalent:
\begin{enumerate}[\normalfont(i)]
    \item $\mathrm{IMG}(f)$ has $\mathrm{(LR)}$;
    \item $\mathrm{IMG}(f)$ is virtually abelian;
    \item $f$ is linearly conjugate (up to a sign) to the powering map $z^2$ or the Chebyshev polynomial $2x^2-1$.
\end{enumerate}
\end{Theorem}

To prove \cref{Theorem: LR quadratic}, we use the (weakly) branch structures of the iterated monodromy groups of post-critically finite hyperbolic quadratic complex polynomials obtained by Bartholdi and Nekrashevych in \cite{Quadratic}, and the results on virtual retracts proved in the first half of the present paper.

Combining our technical results on the finite generation of normal subgroups of iterated monodromy groups (\cref{corollary: G' fg})  with recent results of the first author in \cite{JorgeAV} and of the first author together with Garaialde Ocaña and Uria-Albizuri in \cite{JoneOihana}, we deduce the following: for a quadratic hyperbolic complex polynomial, the closure $\overline{\mathrm{IMG}(f)}$ has complete finitely generated Hausdorff spectrum; {see \cite[Corollary C]{JoneOihana} and \cref{corollary: hspec}. This provides further examples which answer the problems of Klopsch \cite[Question 5.4]{KlopschPhD} and Shalev \cite[Problem 17]{ShalevNewHorizons}.

\subsection*{\textit{\textmd{Notation}}} Groups will be assumed to act on the tree on the right so composition will be written from left to right. We shall use exponential notation for group actions on the tree. We write $H\le_f G$ and $H\trianglelefteq_f G$ for a subgroup $H$ of finite index in $G$ (respectively normal and of finite index in $G$). For $N\trianglelefteq G$ and $H\le G$ such that $G=NH$ and $N\cap H = \{1\}$, we will write $G=N\rtimes H = H\ltimes N$. 

\subsection*{Acknowledgements} 

The first author would like to thank his advisors Gustavo A. Fernández-Alcober and Anitha Thillaisundaram for their constant support. The second author would like to thank his advisor Ashot Minasyan for his feedback and support and Lunds Universitet for its warm hospitality while this work was carried out.

\section{Groups acting on rooted trees}
\label{section: groups acting on rooted trees}

In this section, we introduce basic notions in groups acting on rooted trees and prove some well-known lemmata needed in subsequent sections.

\subsection{Subgroups of $\mathrm{Aut}~T$}

Let $T$ be a spherically homogeneous rooted tree, where each vertex in $T$ has at least 2 descendants. Vertices at distance $n\ge 1$ from the root form the \textit{$n$th level $\mathcal{L}_n$} of $T$. For any $v\in T$, we denote by $T_v$ the subtree rooted at the vertex $v$, formed by the vertices whose unique path to the root of $T$ contains $v$, which is again a spherically homogeneous rooted tree. We may identify vertices in $T$ with finite words $v_1v_2\dotsb v_n$ such that $v_1\in T$ and $v_{i+1}\in T_{v_i}$ for each $1\le i \le n-1$.

We write $\mathrm{Aut}~T$ for the group of automorphisms of $T$ fixing the root. Let us consider a subgroup $G\le \mathrm{Aut}~T$. We write $\mathrm{St}_G(n)$ and $\mathrm{st}_G(v)$ for the pointwise stabilizer in $G$ of the $n$th level of $T$ and for the stabilizer in~$G$ of the vertex $v\in T$ respectively. For $G=\mathrm{Aut}~T$, we shall drop the subscript and write $\mathrm{St}(n)$ and $\mathrm{st}(v)$.

For $g\in \mathrm{Aut}~T$ and $v\in T$, we define the \textit{section of $g$ at $v$} as the unique automorphism $g|_v\in \mathrm{Aut}~T_v$ such that 
$$(vw)^g=v^gw^{g|_v}$$
for any $w\in T_v$. The section map $g\mapsto g|_v$ induces a group homomorphism $\varphi_v:\mathrm{st}(v)\to \mathrm{Aut}~T_v$. We write 
$$G_v:=\varphi_{v}(\mathrm{st}_G(v))$$
and call it the \textit{projection of $G$ at $v$}. For each $n\ge 1$, we further define an isomorphism $\psi_n:\mathrm{St}(n)\to \prod_{v\in \mathcal{L}_n} \mathrm{Aut}~T_{v}$ given by
$$g\mapsto (g|_v)_{v\in \mathcal{L}_n}.$$

\subsection{Branch structures}

For $v\in T$, we define the \textit{rigid vertex stabilizer} $\mathrm{rist}_G(v)$ as the subgroup of $G$ consisting of the elements in $G$ which fix every vertex not in~$T_v$. For distinct vertices at the same level, the corresponding rigid vertex stabilizers commute and have trivial intersection; hence, we shall define for each level $n$ the \textit{rigid level stabilizer} $\mathrm{Rist}_G(n)$ as the direct product $$\mathrm{Rist}_G(n):=\prod_{v\in \mathcal{L}_n}\mathrm{rist}_G(v).$$
Note that $\mathrm{Rist}_G(n)$ is a normal subgroup of $G$. A subgroup $G\le \mathrm{Aut}~T$ is said to be \textit{level-transitive} if it acts transitively on every level of $T$. A level-transitive subgroup whose rigid level stabilizers are infinite (or of finite index in $G$) is called \textit{weakly branch} (respectively \textit{branch}). As level-transitive groups are infinite, branch groups are clearly weakly branch. Branch and weakly branch groups were introduced by Grigorchuk in \cite{Branch}.

We will be using the following standard lemma due to Grigorchuk; see the proof of \cite[Theorem 4]{NewHorizonsGrigorchuk}. We provide a short proof for the convenience of the reader:

\begin{lemma}
\label{lemma: normal subgroups and rists}
    Let $G\le \mathrm{Aut}~T$ and a non-trivial normal subgroup $1\ne N\trianglelefteq G$. Then, for any vertex $v$ moved by $N$, we have  $N\ge \prod_{g\in G}\mathrm{rist}_G(v^g)'$.
\end{lemma}
\begin{proof}
    Let $a,b\in \mathrm{rist}_G(v)$ for some vertex $v\in T$ moved by an element $g\in N$. Then
    $$[a,b]=[[g,a],b]\in [[G,N],G]\le N,$$
    because $a^g$ commutes with both $a$ and $b$, as $g$ moves $v$ so $\mathrm{rist}_G(v)\cap \mathrm{rist}_G(v^g)=1$. Now, as $N$ is normal in $G$ we further get
    \begin{align*}
        N\ge (\mathrm{rist}_G(v)')^G&=\prod_{g\in G}\mathrm{rist}_G(v^g)'.\qedhere
    \end{align*}
\end{proof}

We also need the following slight generalization of a well-known result:

\begin{lemma}
    \label{lemma: rist are not virtually abelian}
    Let $G\le \mathrm{Aut}~T$ be a group such that $\mathrm{rist}_G(v)\ne 1$ for every $v\in \widetilde{T}$, for some infinite subtree $\widetilde{T}\subseteq T$. Then $\mathrm{rist}_G(v)$ is not virtually abelian.
\end{lemma}
\begin{proof}
    First, note that $\mathrm{rist}_G(u)\ne 1$ for every $u\in \widetilde{T}_v$ implies that $\mathrm{rist}_G(v)$ is infinite. Thus, if we consider $K\le_f \mathrm{rist}_G(v)$ of finite index, we still have
    $$\mathrm{rist}_K(u)=\mathrm{rist}_G(u)\cap K\ne 1$$
    for every $u\in \widetilde{T}_v$. Now, arguing as in \cite[Lemma 2.17]{MaximalDominik}, we get 
    $$\mathrm{rist}_K(v)'\ne 1,$$
    as the proof in \cite[Lemma 2.17]{MaximalDominik} does not use the level-transitivity of~$K$. Indeed, it only uses that every rigid vertex stabilizer of $K$ (in $\widetilde{T}$) is non-trivial. Therefore $\mathrm{rist}_G(v)$ is not virtually abelian.
\end{proof}

\subsection{Self-similarity}

Let us assume now that $T$ is the $d$-adic tree for some $d\ge 2$, i.e. the regular rooted tree of degree $d$. We say that $G\le \mathrm{Aut}~T$ is \textit{self-similar} if $g|_v\in G$ for every $g\in G$ and $v\in T$, and \textit{fractal} if $G$ is self-similar, level-transitive and $G_v=G$ for every $v\in T$. For self-similar groups, we have a stronger notion of (weak) branchness. We say that a self-similar, level-transitive group $G$ is \textit{weakly regular branch} if $G$ contains a non-trivial \textit{branching subgroup} $K$, i.e. a subgroup $K$ such that
$$\psi_1(\mathrm{St}_K(1))\ge K\times\dotsb \times K.$$
If $K$ is of finite index, then we say that $G$ is \textit{regular branch}.

\section{(LR) property in weakly branch groups}

First, we generalize a construction of Minasyan in \cite{Ashot} to show that branch groups with torsion do not have property (LR); see the definition below. Then, we characterize virtual retractions among normal subgroups and discuss finite generation of normal subgroups of branch and weakly branch groups. 

\subsection{LR property}

	Let $G$ be a group and let $H\le G$ be a subgroup. Let $K$ be a finite-index subgroup of $G$ that contains $H$. If there exists a homomorphism $f\colon K \to H$ which restricts to the identity on $H$, we say that $H$ is a \emph{retract} of $K$ and a \emph{virtual retract} of~$G$. We write $ H\vr G$ when $H$ is a virtual retract of $G$.

    One can show that $H$ is a retract of $K$ if and only if $ K =  N\rtimes H $ for some $N\trianglelefteq K$ such that $N\cap H =\{1\}$. Indeed, one can take $N:=\ker f$ for the corresponding homomorphism $f:K\to H$. We will use this fact throughout the paper without reference.
    
    We say that a group $G$ has \textit{property (LR)} if every finitely generated subgroup of~$G$ is a virtual retract of $G$.

We recall some relevant statements concerning retracts in groups:
\begin{lemma}\label{lem: props retr} Let $G$ be a group.
    \begin{itemize}
        \item[(i)] Suppose that $H\le G$ and $\varphi\colon G \to \Gamma$ is a homomorphism, injective on $H$ and with $\varphi(H) \vr  \Gamma$. Then $H\vr G$. In particular, if $H\le K\le G$ and $H\vr G$, then $H\vr K$; see \cite[Lemma~3.2.(ii)]{Ashot}. 
        \item[(ii)] If $G$ is residually finite, then $B\vr G$ for all finite subgroups $B$ of $G$; see \cite[Lemma~3.4]{Ashot}.
        \item[(iii)] If $G$ is finitely generated and virtually abelian, then $G$ has (LR); see \cite[Corollary~4.3]{Ashot}.
    \end{itemize}
\end{lemma}
\subsection{Branch groups with torsion}

We use a similar construction as Minasyan in \cite[Example~3.6 and Proposition~3.7]{Ashot} to prove that branch groups with torsion elements do not have (LR):

\begin{lemma}
\label{lemma: LR with torsion}
Let $G\le \mathrm{Aut}~T$ be a finitely generated branch group, which contains a torsion element. Then $G$ does not have $\mathrm{(LR)}$. 
\end{lemma}
\begin{proof}
     As $G$ contains a torsion element, there exists $g\in G$ of order $p$ for some prime $p\ge 2$. Then, there is a vertex $v\in T$ moved by $g$ whose $g$-orbit is precisely
    $$
    \mathcal{O}:=\{v,v^g,\dotsc v^{g^{p-1}}\}.
    $$
    In particular, the $g$-orbit of $v$ is of length $p$. Let us consider the subgroups
    $$
    R:=\prod_{u\in \mathcal{O}}\mathrm{rist}_G(u)
    $$
    and 
    $$
    H:=\left\{\prod_{i=0}^{p-1}x^{g^i}\mid x\in \mathrm{rist}_G(v)\right\}\le R.
    $$
    Observe that $H\cong \mathrm{rist}_G(v)$ is finitely generated. Indeed, as $\mathrm{Rist}_G(n)$ is of finite index in $G$ for every $n\ge 1$ and 
    $$
    \mathrm{Rist}_G(n)=\prod_{u\in \mathcal{L}_n}\mathrm{rist}_G(u),
    $$
    we have that $\mathrm{rist}_G(u)$ must be finitely generated for every $u\in T$.

    Now, let us define the finitely generated subgroup
    $$
    A:=\langle H, g\rangle,
    $$
     where $H\le_f A$ as $A = H \rtimes \langle g\rangle$.
    We proceed by proving that, if $A$ is virtual retract of $G$, we reach a contradiction, and as $A$ is finitely generated, this yields that $G$ does not have (LR). If $A\vr G$, then $$
    A\vr  L \coloneq \left( \prod_{u\in \mathcal{O}}\mathrm{rist}_G(u) \right) \rtimes   \langle g \rangle= R \rtimes \langle g \rangle \le  G,
    $$ 
    by \cref{lem: props retr}\textcolor{teal}{(i)}. Thus, there exist $K \le_f L$ and a normal subgroup $ N\trianglelefteq K$ such that $K = NA$ and $N \cap A = \{1\}$.
 
     We claim that $N$ must be trivial. Indeed, as $N$ is normal in $K$, then if $N$ is not trivial, \cref{lemma: normal subgroups and rists} yields that $N\ge \prod_{g\in K}\mathrm{rist}_K(w^g)'$ for any vertex $w$ moved by $N$. As $R$ is normal and of finite index in $L$, we may assume without loss of generality that $N\le R$. Then, the vertex $w$ may be taken to be a descendant of some $u=v^{g^j}\in \mathcal{O}$. 
   
    If $x\in \mathrm{rist}_G(u)\cap N$, then 
    $$\prod_{i=0}^{p-1}(x^{g^{-j}})^{g^i}\in N\cap A =\{1\} $$
    as $A$ normalizes $N$ and $g\in A$. Therefore $\mathrm{rist}_G(u)\cap N = \{1\}$. However
    $$
        \mathrm{rist}_G(u)\cap N\ge \mathrm{rist}_K(w)'\ne 1
    $$
    as $L$ (and thus any finite-index subgroup of $L$ such as $K$) satisfies the assumption in \cref{lemma: rist are not virtually abelian} on the subtree $T_u$. Therefore $N$ must be trivial and 
    $$H\le_f A = K \le_f L= \left(\prod_{u\in \mathcal{O}}\mathrm{rist}_G(u) \right)\rtimes \langle g \rangle.$$
    In particular, $H\cap\mathrm{rist}_G(v)\ne \{1\}$ since
    $$H\cap\mathrm{rist}_G(v) \le_f \mathrm{rist}_G(v)$$
    and $\mathrm{rist}_G(v)$ is infinite. But $H$ does not intersect $\mathrm{rist}_G(v)$ by construction.
   
\end{proof}

The proof of \cref{lemma: LR with torsion} only works if $G$ contains torsion and for each $v\in T$ there is a finitely generated subgroup $N_v\le \mathrm{rist}_G(v)$ normal in $\mathrm{st}_G(v)$. Indeed, if $G$ does not contain an element of finite order, then the subgroup $H$ is not normalized by~$g$ in the proof of \cref{lemma: LR with torsion}. A natural class of groups to consider is fractal weakly regular branch groups, as then $1\ne N\le \varphi_v(\mathrm{rist}_G(v))\le G$ for every $v\in T$ for~$N$ the maximal normal branching subgroup \cite[Lemma 1.3]{Normal}. However, if $G$ is not regular branch, then $N$ is of infinite index and it may not be finitely generated. We shall see that if $N$ is finitely generated, then we do not need $G$ to contain torsion anymore. In particular, this will allow us to study the (LR) property in torsion-free weakly branch groups.

\subsection{Virtual retracts of normal subgroups}

Let us characterize when the normal subgroups of a weakly branch group are virtual retracts:

\begin{lemma}
\label{lemma: normal subgroups when are they retracts}
    Let $G\le \mathrm{Aut}~T$ be a weakly branch group and $1\ne N\trianglelefteq G$ a non-trivial normal subgroup of $G$. Then $N\vr G$ if and only if $N\le_{\mathrm{f}}G$.
\end{lemma}
\begin{proof}
If $N$ is of finite index, then clearly $N\vr G$. Let $N$ be a non-trivial normal subgroup of infinite index in $G$. Assume by contradiction that $N$ is a virtual retract of $G$. Then, there exist subgroups $L$ and $K$ of $G$ such that $L\trianglelefteq K \le_f G$ and $K= L\times N$  has finite index in $G$. As $N$ is normal in $G$ and  $L$ is normal in $K$, by \cref{lemma: normal subgroups and rists} we have
    $$N\ge \mathrm{Rist}_G(n)'\quad\text{and}\quad L\ge \prod_{g\in K}\mathrm{rist}_K(v^g)'$$
    for some $n\ge 1$ and a vertex $v\in T$ moved by $L$. We may assume without loss of generality that $v$ is at level $n$ or below. Since $K$ is of finite index in $G$ and~$G$ is weakly branch, for any $u\in T$ we have
    $$\mathrm{rist}_K(u)=\mathrm{rist}_G(u)\cap K\ne 1$$
    as $\mathrm{rist}_G(u)$ is infinite. Then, by \cref{lemma: rist are not virtually abelian} we have 
    $$\mathrm{rist}_K(u)'\ne 1$$
    for every $u\in T$. Therefore 
    $$L\cap N\ge \mathrm{Rist}_G(n)'\cap \mathrm{rist}_K(v)'=\mathrm{rist}_K(v)'\ne 1,$$
    which yields a contradiction to $K= N\times L$.
\end{proof}

In particular, \cref{lemma: normal subgroups when are they retracts} yields a sufficient condition for $G$ to not have (LR):

\begin{corollary}
\label{corollary: normal LR}
    Let $G\le \mathrm{Aut}~T$ be a weakly branch group admitting a finitely generated normal subgroup $1\ne N\trianglelefteq G$ of infinite index in $G$. Then $G$ does not have $\mathrm{(LR)}$.
\end{corollary}

\subsection{Finitness properties of normal subgroups}

For a finitely generated branch group, every proper quotient is virtually abelian and hence finitely presented. Thus, every normal subgroup is normally finitely generated. Here, we show that they are all in fact finitely generated. This has been recently been proved by Francoeur, Grigorchuk, Leemann and Nagnibeda in \cite{FGLN}.

\begin{proposition}
\label{proposition: normal subgroups are fg}
    Let $G\le \mathrm{Aut}~T$ be a finitely generated branch group. Then every normal subgroup $1\ne N\trianglelefteq G$ is finitely generated.
\end{proposition}
\begin{proof}
    We only need to show that $\mathrm{Rist}_G(n)'$ is finitely generated for every $n\ge 1$. Indeed, by \cref{lemma: normal subgroups and rists}, every normal subgroup $1\ne N\trianglelefteq G$ contains $\mathrm{Rist}_G(n)'$ for some $n\ge 1$. As $G/\mathrm{Rist}_G(n)'$ is finitely generated and virtually abelian, the quotient $N/\mathrm{Rist}_G(n)'$ is finitely generated. Thus, if $\mathrm{Rist}_G(n)'$ is finitely generated, we obtain that $N$ is itself finitely generated.

    Therefore, let us prove that $\mathrm{Rist}_G(n)'$ is finitely generated. Let us fix $n\ge 1$ and $v\in \mathcal{L}_n$. Let $u\in T_v$ be a vertex moved by some $g\in\mathrm{rist}_G(v)$ at some level $k\ge n+1$ in $T$. Then, for every $h\in \mathrm{rist}_G(u)\le \mathrm{rist}_G(v)$ we have
    $$[g,h]=({h^{-1}})^{g} h \in \mathrm{st}_G(u).$$
    As $({h^{-1}})^{g}\in \mathrm{rist}_G(u)^g=\mathrm{rist}_G(u^g)$ and $\mathrm{rist}_G(u)\cap \mathrm{rist}_G(u^g)=1$, we get
    $$[g,h]|_u=h|_u.$$
    In other words
    \begin{equation}\label{eq: eq1}
        \varphi_u(\mathrm{rist}_G(v)'\cap \mathrm{st}_G(u))\ge \varphi_u(\mathrm{rist}_G(u)).
    \end{equation}
    
    Now, as $\mathrm{Rist}_G(n)$ are of finite index for all $n\ge 1$, they are finitely generated. Therefore $\mathrm{rist}_G(w)$ is finitely generated for every $w\in T$. Let $R$ be a finite generating set of $\mathrm{rist}_G(u)$. Then $\mathrm{rist}_G(u)'$ is normally finitely generated in $\mathrm{rist}_G(u)$ as
    $$\mathrm{rist}_G(u)'=\langle [r,t]\mid r,t\in R\rangle ^{\mathrm{rist}_G(u)}=\langle [r,t]\mid r,t\in R\rangle^{\langle R\rangle}.$$
    Thus, by \Cref{eq: eq1}, there exists a finite set $A\subseteq \mathrm{rist}_G(v)'$ such that $A\subseteq \mathrm{st}_G(u)$ and $\varphi_u(A) = \varphi_u(R)$. Denoting $S := \{[r,t]\mid r,t\in R\}$, we get
    $$\varphi_u(\mathrm{rist}_G(v)'\cap \mathrm{st}_G(u))\ge \varphi_u(\langle S\rangle^{\langle A\rangle})= \langle \varphi_u(S)\rangle^{\langle \varphi_u(R)\rangle}=\varphi_u(\mathrm{rist}_G(u)').$$
    Injectivity of $\varphi_u$ restricted to $\mathrm{rist}_G(u)$ implies that
    $$\mathrm{rist}_G(v)'\ge \mathrm{rist}_G(v)'\cap \mathrm{st}_G(u)\ge  \langle S\rangle^{\langle A\rangle}=\mathrm{rist}_G(u)'.$$
    As $G$ is level-transitive, for any $z\in \mathcal{L}_k$, there is an element $g_z\in G$ such that $u^{g_z}=z$. Therefore, if we write $X:=\bigcup_{z\in \mathcal{L}_k} S^{g_z}\cup A^{g_z}$, we get
    \begin{align*}
        \mathrm{Rist}_G(n)'&=\prod_{g\in G}\mathrm{rist}_G(v^g)'=\prod_{g\in G}(\mathrm{rist}_G(v)')^g\ge \langle X\rangle\\
        &\ge \prod_{z\in \mathcal{L}_k}(\mathrm{rist}_G(u)')^{g_z}= \prod_{z\in \mathcal{L}_k}\mathrm{rist}_G(u^{g_z})'=\mathrm{Rist}_G(k)'.
    \end{align*}
    Arguing as above, the quotient $\mathrm{Rist}_G(n)'/\mathrm{Rist}_G(k)'$ is finitely generated. Therefore, there is some finite set $Y\subset \mathrm{Rist}_G(n)'$ such that $\mathrm{Rist}_G(n)'=\langle X\cup Y\rangle$, so $\mathrm{Rist}_G(n)'$ is finitely generated.
\end{proof}

\begin{remark}
\label{remark: K}
    The proof of \cref{proposition: normal subgroups are fg} works for weakly branch groups if for every $n\ge 1$, both $\mathrm{Rist}_G(n)$ is finitely generated and every subgroup of $G/\mathrm{Rist}_G(n)$ is finitely generated (equivalently polycyclic). Indeed, then $G/\mathrm{Rist}_G(n)'$ would be polycyclic-by-polycyclic and thus polycyclic itself. In particular, if $G$ is weakly regular branch over a normal subgroup $K$ such that every subgroup of $G/K$ is finitely generated, then every normal subgroup $1\ne N\trianglelefteq G$ is finitely generated in $G$ if and only if $K$ is finitely generated. This follows from the same proof as above with $K$ playing the role of $\mathrm{rist}_G(v)$ for each $v\in T$. Thus, the main difficulty lies in showing that $K$ is finitely generated. We shall address this issue for iterated monodromy groups of periodic quadratic polynomials.
\end{remark}

\subsection{(VRC) property in just-infinite branch groups}

We say that a group $G$ has \textit{property (VRC)} if every cyclic subgroup of $G$ is a virtual retract of $G$.

The following property will be key in the study of the (VRC) property in just-infinite groups:

\begin{lemma}[{see \cite[Corollary~3.8]{J-A}}]
\label{lem:retr_onto_virt_ab->homom}  
    Let $G$ be a finitely generated group, and let $H \vr G$ be a finitely generated virtually abelian subgroup. Then, there is a virtually abelian quotient $G/N$ such that $G \to G/N$ is injective on $H$.
\end{lemma}

Let us characterize the (VRC) property among just-infinite branch groups:

\begin{lemma}
\label{lemma: just-infinite VRC}
    Let $G\le \mathrm{Aut}~T$ be just-infinite branch group. Then $G$ has $\mathrm{(VRC)}$ if and only if $G$ is torsion.
\end{lemma}
\begin{proof}
    As $G$ is residually finite, all cyclic subgroups generated by a torsion element are virtual retracts by \cref{lem: props retr}\textcolor{teal}{(ii)}. For the converse, note that if $\langle g\rangle$ is a virtual retract of $G$ then there exists a virtually abelian quotient $G/N$ such that $\langle g\rangle \hookrightarrow G/N$ by \Cref{lem:retr_onto_virt_ab->homom}. Since~$G$ is branch, it is not virtually abelian by \Cref{lemma: rist are not virtually abelian}, so $N$ must be non-trivial. However, as $G$ is just-infinite, $G/N$ must be finite and thus $g$ is torsion.
\end{proof}

As (LR) implies (VRC), \cref{lemma: just-infinite VRC} yields the last piece needed to prove that finitely generated branch groups do not have (LR):

\begin{proof}[Proof of \cref{Theorem: branch not LR}]
    If $G$ is not just-infinite, then $G$ contains a normal subgroup $1\ne N\trianglelefteq G$ of infinite index. Thus, by \cref{proposition: normal subgroups are fg}, the subgroup $N$ is finitely generated and by \cref{corollary: normal LR} the group $G$ does not have (LR). Thus, let us assume that $G$ is just-infinite. Then, if $G$ is not torsion it does not have (VRC) and thus not (LR) either. If $G$ is torsion, then $G$ does not have (LR) by \cref{lemma: LR with torsion}.
\end{proof}

\section{Iterated monodromy groups of quadratic polynomials}
\label{section: iterated monodromy groups}

In this last section, we specialize to iterated monodromy groups. First, we introduce iterated monodromy groups and several of their properties. Then, we study the finite-generation of normal subgroups for periodic polynomials and show an application to the finitely generated Hausdorff spectrum of their closure. Lastly, we prove \cref{Theorem: LR quadratic}.

\subsection{Iterated monodromy groups}

Let $f:\widehat{\mathbb{C}}\to \widehat{\mathbb{C}}$ be a rational function of degree $d\ge 2$ on the Riemann sphere $\widehat{\mathbb{C}}$. For any $z\in \widehat{\mathbb{C}}$, we say that $z$ is \textit{periodic} if there is $n\ge 1$ such that $f^n(z)=z$, and \textit{preperiodic} if $z$ is not periodic and there is $n>m\ge 1$ such that $f^n(z)=f^m(z)$. The map $f$ is said to be \textit{post-critically finite} if every critical point $c\in C_f$ is either periodic or preperiodic. Equivalently $f$ is post-critically finite if the postcritical set
$$P_f:=\bigcup_{n\ge 1}\{f^n(c)\mid c\in C_f\}$$
is finite.  If $f$ is post-critically finite, then $f$ induces an unramified branch covering 
$$f:\widehat{\mathbb{C}}\setminus f^{-1}(P_f)\to \widehat{\mathbb{C}}\setminus P_f.$$
Let $z \in \widehat{\mathbb{C}} \setminus P_f$, a loop $\gamma$ starting and ending at $z$ and $x \in f^{-1}(z)$. Let us consider $\gamma_x$ the unique lift of $\gamma$ starting at $x$. As $f(\gamma_x)=\gamma$, the endpoint of $\gamma_x$ is another point in $f^{-1}(z)$, and we obtain a well-defined action of the fundamental group $\pi_1(\widehat{\mathbb{C}}\setminus P_f,z)$ on $f^{-1}(z)$ via
$$x^{\gamma}:=\gamma_x(1),$$
where $\gamma_x(1)$ denotes the endpoint of $\gamma_x$. This action is called the \textit{monodromy action} of $f$. Similarly, for any $n\ge 1$ we obtain an unramified branch covering
$$f^n:\widehat{\mathbb{C}}\setminus f^{-n}(P_f)\to \widehat{\mathbb{C}}\setminus P_f$$
and a well-defined action of the fundamental group $\pi_1(\widehat{\mathbb{C}}\setminus P_f,z)$ on $f^{-n}(z)$. These actions are coherent, so we obtain an action of the fundamental group $\pi_1(\widehat{\mathbb{C}}\setminus P_f,z)$ by automorphisms on the tree of preimages
$$T:=\{z\}\cup\bigsqcup_{n\ge 1} f^{-n}(z).$$
As $z \notin P_f$, the set $f^{-n}(z)$ has cardinality $d^n$ for each $n\ge 1$. Thus $T$ is the $d$-adic tree.

The quotient of the fundamental group $\pi_1(\widehat{\mathbb{C}}\setminus P_f,z)$ by the kernel of its action on the tree of preimages $T$ above is isomorphic to a self-similar subgroup $\mathrm{IMG(f)}$ of $\mathrm{Aut}~T$. We call the group $\mathrm{IMG}(f)$ the \textit{iterated monodromy group} of $f$. As $P_f$ is finite, the space $\widehat{\mathbb{C}}\setminus P_f$ is path connected and $\mathrm{IMG}(f)$ is independent of the choice of the base point $z\notin P_f$, up to conjugation in $\mathrm{Aut}~T$. The iterated monodromy group is a discrete analogue of the geometric iterated Galois group of $f$, given by the iterated monodromy action of the \'{e}tale fundamental group $\pi_1^{\text{\'{e}t}}(\widehat{\mathbb{C}}\setminus P_f, z)$ on $T$ or, equivalently, by the Galois action of the Galois group $\mathrm{Gal}(\mathbb{C}_\infty(f,t)/\mathbb{C})$ of the splitting field $\mathbb{C}_\infty(f,t)$ of all the iterates of $f$ on the tree of preimages of a transcendental element $t$ over $\mathbb{C}$.

We shall now focus on complex polynomials. Actually, what we will be doing works generally for \textit{topological polynomials}, i.e. post-critically finite branch coverings of the Riemann sphere such that $f^{-1}(\infty)=\{\infty\}$, i.e. $\infty$ is a completely ramified super attracting fixed point for $f$. The obstructions for a topological polynomial with a hyperbolic orbifold to be a complex polynomial are known to be Levy cycles \cite{Douady}. Those topological polynomials with euclidean orbifolds are not obstructed and they are precisely the powering maps and the Chebyshev polynomials; see \cite{Douady} for a complete classification of the rational functions with euclidean orbifolds.

\subsection{Kneading automata}

A \textit{multiset of permutations} is a map $I\to \mathrm{Sym}(X)$, which we denote $\{\sigma_i\}_{i\in I}$. The \textit{cycle diagram}  of $\{\sigma_i\}_{i\in I}$ is the oriented CW-complex whose 0-cells are the elements in $X$, and whose 2-cells are given by (ordered) polygons corresponding to the cycles in the permutations in $\{\sigma_i\}_{i\in I}$. Note that the 2-cells only intersect along the 0-cells. We say that $\{\sigma_i\}_{i\in I}$ is \textit{tree-like} if its cycle diagram is contractible.

A self-similar group $G=\langle S\rangle\le \mathrm{Aut}~T$ is a \textit{kneading automata group} if $G$ is an automata group (i.e. the sections of the generators in $S$ at the first level of $T$ are in $S\cup \{1\}$) such that:
\begin{enumerate}[\normalfont(i)]
    \item for each generator $g\in S$, there is a unique generator $h\in S$ and a unique vertex $v\in \mathcal{L}_1$ such that $h|_v=g$;
    \item for each generator $g\in S$ and every cycle $(v_1\,\dotsb v_r)$ of the action of $g$ on $\mathcal{L}_1$ the section $g|_{v_i}$ is non-trivial for at most one vertex $v_i$ among $v_1,\dotsc,v_r$;
    \item the multiset of permutations induced by the actions of the generators in $S$ on $\mathcal{L}_1$ is tree-like.
\end{enumerate}

A well-known property of kneading automata groups that we shall use is that the product of the generators in any order is a level-transitive element \cite{SelfSimilar}.

Let $f\in \mathbb{C}[x]$ be a complex polynomial of degree $d\ge 2$. The group $\mathrm{IMG}(f)$ has a generating set $\{g_p\mid p\in P_f\}$ indexed by the complex post-critical points. The group $\mathrm{IMG}(f)=\langle g_p\mid p\in P_f\rangle$ is an instance of a kneading automata group \cite{SelfSimilar}. In fact, more is true. For any $p\in P_f$, the non-trivial sections of $g_p$ at the first level are the generators $g_q$ such that $q\in P_f\cap f^{-1}(p)$, and the action of $g_p$ on the first level is given by cycles of length corresponding to the ramification degree of the preimages on $f^{-1}(p)$. The reader is refered to \cite{SelfSimilar} for a detailed exposition in terms of invariant spiders.

\subsection{Structural properties}

Let $f\in \mathbb{C}[x]$ be a quadratic polynomial. Then $f$ has a unique critical point $c$. We say that $f$ is \textit{periodic} if $c$ is periodic and \textit{preperiodic} if $c$ is preperiodic.

We need several structural results for the iterated monodromy groups of quadratic complex polynomials. We shall restate several results for kneading automata groups in \cite{Quadratic} in the language of iterated monodromy groups:

\begin{proposition}[{see \cite[Theorems 3.12 and 4.10 and Proposition 3.13]{Quadratic}}]
\label{proposition: branch structures}
    Let $f\in \mathbb{C}[x]$ be a quadratic polynomial and $G:=\mathrm{IMG}(f)$. Then
    \begin{enumerate}[\normalfont(i)]
        \item if $f$ is preperiodic, either $f$ is linearly conjugate to the Chebyshev polynomial $2z^2-1$ or $G$ is regular branch;
        \item if $f$ is periodic, either $f$ is linearly conjugate to the powering map $z^2$ or $G$ is torsion-free, weakly regular branch over $G'$ and $G/G'$ is (non-trivial) free abelian.
    \end{enumerate}
\end{proposition}

\subsection{Normal subgroups}

In \cite{FPP}, the first author together with Radi generalized the notion of exceptional set of Makarov and Smirnov \cite{MakarovSmirnov}. Given a complex polynomial $f\in \mathbb{C}[x]$ of degree $d\ge 2$, we say that $\Upsilon\subseteq P_f$ is \textit{critically exceptional (for $f$)}  if 
$$\Upsilon=f^{-1}(\Upsilon)\setminus ((C_f\cup P_f)\setminus \Upsilon).$$
The union of critically exceptional sets is critically exceptional and thus there is a maximal critically exceptional $\Upsilon_f$. It is proved in \cite{FPP} that $\#(\Upsilon_f\cap \mathbb{C})\le 2$. We say that~$f$ is \textit{critically exceptional} if $\Upsilon_f\cap \mathbb{C}\ne \emptyset$.

We show that many non critically exceptional iterated monodromy groups satisfy a strong finiteness property:

\begin{lemma}
    \label{lemma: G' fg}
    Let $f\in \mathbb{C}[x]$ be a complex polynomial of degree $d\ge 2$, and let us write $G:=\mathrm{IMG}(f)$. Assume further that
    \begin{enumerate}[\normalfont(i)]
        \item $f$ is not critically exceptional, i.e. $\Upsilon_f\cap \mathbb{C}=\emptyset$;
        \item $G$ is weakly regular branch over $G'$.
    \end{enumerate}
    Then, every normal subgroup $1\ne N\trianglelefteq G$ is finitely generated.
\end{lemma}
\begin{proof}
    First note that by \cref{remark: K}, it is enough to show that $G'$ is finitely generated. Let $S:=\{g_1,\dotsc, g_r\}$ be a generating set of $G$. As conditions (i)-(iii) in the definition of a kneading automata group are stronger than their counterparts for the geometric iterated Galois group $\overline{\mathrm{IMG}(f)}$ (i.e. those in \cite[Remark 5.2]{FPP}), we may argue as in the proof of \cite[Theorem 5.6]{FPP} to obtain that there exists $N\ge 1$ such that for each $v\in \mathcal{L}_N$ and $1\le i\le r$ there is $s_i^v\in \mathrm{st}_{G'}(v)$ such that
    $$s_i^v|_v=g_i.$$
    Now, as $G=\langle S\rangle$, we have
    $$G'=\langle [s,t]\mid s,t\in S\rangle^G=\langle [g_i,g_j]\mid 1\le i,j\le r\rangle^{\langle S\rangle}.$$
    As $G$ is weakly regular branch over $G'$, let $k_{ij}^v\in \mathrm{rist}_{G'}(v)$ be such that
    $$k_{ij}^v|_v=[g_i,g_j]$$
    for each $1\le i,j\le r$ and $v\in \mathcal{L}_N$. We write $K_N\le \mathrm{Rist}_G(n)$ for the subgroup such that
    $$\psi_N(K_N)=G'\times\dotsb \times G'.$$
    Then
    \begin{align*}
        G'&\ge \langle s_i^v,k_{ij}^v\mid 1\le i,j\le  r\text{ and }v\in \mathcal{L}_N\rangle\\
        &\ge \prod_{v\in \mathcal{L}_N}\langle k_{ij}^v \mid 1\le i,j\le  r\rangle^{\langle s_i^v \mid 1\le i\le r\rangle}\\
        &= K_N.
    \end{align*}
    Now, note that $G'/K_n$ is finitely generated. Indeed, as
    $$\mathrm{St}_{G'}(N)/K_N\hookrightarrow G/G'\times\dotsb \times G/G'$$
    is a subgroup of a finitely generated abelian group, it is finitely generated itself. Therefore, $G'/K_N$ is finitely generated as $G'/\mathrm{St}_{G'}(N)$ is finite. Let $T\subset G'$ be a finite set such that 
    $$G'/K_N=\langle tK_N\mid t\in T\rangle.$$
    We finally conclude that
    \begin{align*}
        G'&=\langle t,s_i^v,k_{ij}^v\mid t\in T,~ 1\le i,j\le r\text{ and }v\in \mathcal{L}_N\rangle.\qedhere
    \end{align*}
    
\end{proof}

The conditions of \cref{lemma: G' fg} are satisfied by  periodic quadratic polynomials. In fact, if $f$ is a periodic quadratic polynomial, for every pair $x,y\in P_f$ there is some $n\ge 0$ such that $f^n(x)=y$. Hence, as $f(\Upsilon_f)\subseteq \Upsilon_f$, either $\Upsilon_f=P_f$ or $\Upsilon_f=\emptyset$. However, the set $f^{-1}(c)$ contains a point $z$ outside $C_f\cup P_f$, so $c\notin \Upsilon_f$. Thus, we must have $\Upsilon_f=\emptyset$. Lastly, by \cref{proposition: branch structures}, the group $\mathrm{IMG}(f)$ is either infinite cyclic if $f$ is linearly conjugate to $z^2$ or weakly regular branch over $\mathrm{IMG}(f)'$. In either case, we obtain the following:

\begin{corollary}
\label{corollary: G' fg}
    Let $f\in \mathbb{C}[x]$ be a periodic quadratic polynomial and $G:=\mathrm{IMG}(f)$. Then, every normal subgroup $1\ne N\trianglelefteq G$ is finitely generated.
\end{corollary}

\subsection{Projection-invariant subgroups}

We say that $G\le \mathrm{Aut}~T$ is \textit{projection-invariant} if $G_v\le G$ for every $v\in T$. Under the assumptions of \cref{lemma: G' fg}, the group $G:=\mathrm{IMG}(f)$ does not contain any projection-invariant proper non-trivial normal subgroup.

Indeed, if $1\ne N\trianglelefteq G$, then there is a vertex $v\in T$ such that $N_v\ge G''$, as $N\ge \mathrm{Rist}_G(n)'\ge G''\times\dotsb \times G''$ for some $n$ large enough by \cref{lemma: normal subgroups and rists}. As in the proof of \cref{lemma: G' fg}, we get $(G')_u=G$ for some $u\in \mathcal{L}_N$ and thus 
$$(G'')_{uu}=((G'')_u)_u\ge (((G')_u)')_u=(G')_u=G.$$
Therefore $N_{vuu}=G$.

This extends previous results of Adams in \cite{Ophelia}, from open non-trivial normal subgroups of $\overline{G}$ to closed non-trivial normal subgroups of $\overline{G}$. Projection-invariant normal subgroups of $\overline{G}$ are closely related to the dynamical properties of the polynomial $f$; see \cite{Ophelia}.

\subsection{Finitely generated Hausdorff spectrum}

\cref{corollary: G' fg} has further applications to the finitely generated Hausdorff spectrum of weakly branch groups.

Let $G\le \mathrm{Aut}~T$ be a closed subgroup. The filtration of level stabilizers induces a metric $d_G(\cdot,\cdot)$ in $G$, given by
$$d_G(g,h):=\inf_{n\ge 0}\{|G:\mathrm{St}_G(n)|^{-1} \mid gh^{-1}\in \mathrm{St}_G(n)\}$$
for any $g,h\in G$. Then, we may define a normalized Hausdorff dimension $\mathrm{hdim}_G(\cdot)$ in the compact metric space $(G,d_G)$. For $b\ge 1$, we define the \textit{b-bounded Hausdorff spectrum} of $G$ as
$$\mathrm{hspec}_{b}(G):=\{\mathrm{hdim}_G(H)\mid H\le G\text{ and } d(H)<b\},$$
where $d(H)$ denotes the minimum number of topological generators of $H$.

The first author, together with Garaialde Ocaña and Uria-Albizuri, proved in \cite{JoneOihana} that if $G$ is a fractal group weakly regular branch over a finitely generated normal subgroup $N$ with $\mathrm{hdim}_{\overline{G}}(\overline{N})=1$, then 
$$\mathrm{hspec}_b(\overline{G})=[0,1]$$
for some $b\ge 1$. In fact, by a result of the first author \cite[Theorem C]{JorgeAV}, the assumption $\mathrm{hdim}_{\overline{G}}(\overline{N})=1$ is always satisfied if $G$ is self-similar. Hence, by \cref{proposition: branch structures} and \cref{corollary: G' fg}, the iterated monodromy group of a periodic quadratic polynomial satisfies these assumptions and its bounded Hausdorff spectrum is complete:

\begin{corollary}
\label{corollary: hspec}
    Let $f\in \mathbb{C}[x]$ be a periodic quadratic polynomial. Then, there exists $b\ge 1$ such that
    $$\mathrm{hspec}_b(\overline{\mathrm{IMG}(f)})=[0,1].$$
\end{corollary}

\cref{corollary: hspec} yields the first infinite family of weakly branch groups with complete bounded Hausdorff spectrum, generalizing the analogous result for the Basilica group $\mathrm{IMG}(z^2-1)$ given in \cite{JoneOihana}.

\subsection{(LR) property}

Now we prove \cref{Theorem: LR quadratic}, namely that (LR) property characterizes the powering map and the Chebyshev polynomial among quadratic complex polynomials:

\begin{proof}[Proof of \cref{Theorem: LR quadratic}]
    First, note that if $f$ is linearly conjugate to the powering map, then $\mathrm{IMG}(f)\cong \mathbb{Z}$. Secondly, if $f$ is linearly conjugate to the Chebyshev polynomial, then it is well-known that $\mathrm{IMG}(f)\cong D_\infty$; see for instance \cite{SelfSimilar}. In both cases, the corresponding iterated monodromy groups are virtually abelian and thus have (LR) by \Cref{lem: props retr}\textcolor{teal}{(iii)}.

    Now, if $f$ is preperiodic and not linearly conjugate to the Chebyshev polynomial, then $\mathrm{IMG}(f)$ is branch by \cref{proposition: branch structures}\textcolor{teal}{(i)}. Hence, the group $\mathrm{IMG}(f)$ does not have (LR) by \cref{Theorem: branch not LR}.

    Lastly, let us assume that $f$ is periodic and not linearly conjugate to the powering map. Then $\mathrm{IMG}(f)'$ is finitely generated by \cref{corollary: G' fg}. Furthermore, the commutator subgroup is normal of infinite index in $\mathrm{IMG}(f)$ by \cref{proposition: branch structures}\textcolor{teal}{(ii)}. Therefore, the group $\mathrm{IMG}(f)$ does not have (LR) in this case neither by \cref{corollary: normal LR}.
\end{proof}




\bibliographystyle{unsrt}

\begin{thebibliography}{1}

\bibitem{Ophelia}
O. Adams, Semiconjugacy and self-similar subgroups of pfIMGs, arXiv preprint: 2508.12122.


\bibitem{Quadratic}
L. Bartholdi and V.\,V Nekrashevych, Iterated monodromy groups of quadratic polynomials, I, \textit{Groups Geom. Dyn.} \textbf{2} (2008), 309--336.

\bibitem{Thurston}
L. Bartholdi and V. Nekrashevych, Thurston equivalence of topological polynomials, \textit{Acta Math.} \textbf{197} (2006), 1--51.

\bibitem{RW2}
L. Bartholdi, V.\,A. Kaimanovich and V.\,V. Nekrashevych, On amenability of automata groups, \textit{Duke Math. J.} \textbf{154(3)} (2010), 575--598.

\bibitem{Normal}
L. Bartholdi, O. Siegenthaler and P. Zalesskii, The congruence subgroup problem for branch groups, \textit{Isr. J. Math.} \textbf{187} (2012), 419--450.

\bibitem{RW1}
L. Bartholdi and V. Virág, Amenability via random walks, \textit{Duke Math. J.} \textbf{130(1)} (2005), 39--56.

\bibitem{Brid-Wilt}
M.\,R. Bridson and H. Wilton, Subgroup separability in residually free groups, \textit{Math. Z.} \textbf{260 (1)} (2008), 25--30.

\bibitem{Bridy}
A. Bridy, R. Jones, G. Kelsey, and R. Lodge, Iterated monodromy groups of rational functions and periodic points over finite fields, \textit{Math. Ann.} \textbf{390(1)} (2024), 439--475.


\bibitem{Douady}
A. Douady and J.\,H. Hubbard, A proof of Thurston’s topological characterization of rational functions, \textit{Acta Math.} \textbf{171(2)} (1993), 263--297.

\bibitem{Erschler1}
A. Erschler, Boundary behavior for groups of subexponential growth, \textit{Ann. Math.} \textbf{160} (2004), 1183--1210.

\bibitem{Erschler2}
A. Erschler and T. Zheng, Growth of periodic Grigorchuk groups, \textit{Invent. Math.} \textbf{219} (2020), 1069--1155.


\bibitem{JorgeAV}
J. Fariña-Asategui, On a question of Abért and Virág, to appear in \textit{Proc. Amer. Math. Soc.}.

\bibitem{JPHT}
J. Fariña-Asategui, P.-H. Leemann, T. Nagnibeda, Non-closed subgroups of weakly branch groups, arXiv preprint: 2511.15277.

\bibitem{JoneOihana}
J. Fariña-Asategui, O. Garaialde Ocaña and J. Uria-Albizuri, On finitely generated Hausdorff spectra of groups acting on rooted trees, arXiv preprint: 2504.11948.

\bibitem{FPP}
J. Fariña-Asategui and S. Radi, Fixed-point proportion of geometric iterated Galois groups, arXiv preprint: arXiv preprint: 2601.16173.

\bibitem{MaximalDominik}
D. Francoeur, On maximal subgroups of infinite index in branch and weakly branch groups, \textit{J. Algebra} \textbf{560}
(2020), 818--851.

\bibitem{FGLN}
D. Francoeur, R.\,I. Grigorchuk, P.-H. Leemann, and T. Nagnibeda, On the structure of finitely generated subgroups of branch groups, to appear in \textit{J. Comb. Algebra}.


\bibitem{FL}
D. Francoeur and P.-H. Leemann, Subgroup induction property for branch groups, \textit{J. Fractal Geom.} \textbf{12 (1)} (2025), 175--207.

\bibitem{Gar16}
 A. Garrido, Abstract commensurability and the Gupta-Sidki group, \textit{Groups Geom. Dyn.} \textbf{10 (2)} (2016), 523--543.

 \bibitem{ABC}
C. Gratton, K. Nguyen and T.\,J. Tucker, ABC implies primitive prime divisors in arithmetic dynamics, \textit{Bull. London Math. Soc.} \textbf{45 (6)} (2013), 1194--1208.

\bibitem{Branch}
R.\,I. Grigorchuk, Branch groups, \textit{Math. Notes} \textbf{67 (6)} (2000), 852--858.




\bibitem{NewHorizonsGrigorchuk}
R.\,I. Grigorchuk, Just infinite branch groups, in:
\textit{New Horizons in Pro-p Groups}, Birkhäuser Boston, MA \textbf{1}, 2000.

\bibitem{GrigorchukBurnside}
R.\,I. Grigorchuk, On Burnside’s problem for periodic groups, \textit{Funktsional. Anal. i Prilozhen} \textbf{14 (1)} (1980), 53–54.

\bibitem{GrigorchukMilnor}
R.\,I. Grigorchuk, On Milnor's problem on group growth, \textit{Soviet Math. Dokl.} \textbf{28} (1983), 23--26.



\bibitem{GW}
R.\,I. Grigorchuk and J.\,S. Wilson, A structural property concerning abstract commensurability of subgroups, \textit{J. London Math. Soc.} \textbf{68 (3)} (2003), 671--682.

\bibitem{Haglund}
F. Haglund, Finite index subgroups of graph products, \textit{Geom. Dedicata} \textbf{135 (1)} (2008), 167--209.

\bibitem{Hall}
M. Hall Jr., Subgroups of finite index in free groups, \textit{Canadian J. Math.} \textbf{1} (1949), 187--190.

\bibitem{JonesAMS}
R. Jones, Fixed-point-free elements of iterated monodromy groups, \textit{Trans. Amer. Math. Soc.} \textbf{367 (3)} (2015), 2023--2049.

\bibitem{JonesComp}
R. Jones, Iterated Galois towers, their associated martingales, and the $p$-adic Mandelbrot set, \textit{Compos. Math.} \textbf{143 (5)} (2007), 1108--1126. 

\bibitem{JonesLMS}
R. Jones, The density of prime divisors in the arithmetic dynamics of quadratic polynomials, \textit{J. Lond. Math. Soc.} \textbf{78 (2)} (2008), 523--544.

\bibitem{Juul}
J. Juul, P. Kurlberg, K. Madhu, and T.\,J. Tucker, Wreath products and proportions of periodic points, \textit{Int. Math. Res. Not. IMRN} \textbf{13} (2016), 3944--3969.



\bibitem{KlopschPhD}
B. Klopsch, \textit{Substitution Groups, Subgroup Growth and Other Topics}, D.Phil. Thesis, University of Oxford (1999).


\bibitem{Hdim}
C. Leung and C. Petsche, The Minkowski dimension of the image of an arboreal Galois representation, arXiv preprint: 2512.18825.

\bibitem{Long-Reid}
D.\,D. Long and A.\,W. Reid, Subgroup separability and virtual retractions of groups, \textit{Topology} \textbf{47 (3)} (2008), 137--159.

\bibitem{MakarovSmirnov}
N. Makarov and S. Smirnov, On “thermodynamics” of rational maps I. Negative spectrum, \textit{Commun. Math. Phys.} \textbf{211(3)} (2000), 705--743.

\bibitem{Malcev}
A.\,I. Mal'cev, On homomorphisms onto finite groups, \textit{Ivanov. Gos. Ped. Inst. Ucen. Zap.} \textbf{18} (1958), 49--60. English translation: \textit{Transl., Ser. 2, Am. Math. Soc.} \textbf{119} (1983), 67--79.


\bibitem{J-A}
J. Merladet Uriguen and A. Minasyan, Virtual retractions in free constructions, to appear in \textit{Trans. Amer. Math. Soc.}.

\bibitem{Ashot}
A. Minasyan, Virtual retraction properties in groups, \textit{Int. Math. Res. Not. (IMRN)} \textbf{2021 (17)} (2019), 13434--13477. 

\bibitem{SelfSimilar}
V. Nekrashevych, Self-similar groups, \textit{American Mathematical Society, Providence, RI} 117 (2005), xii--231.

\bibitem{Odoni1}
R.\,W.\,K. Odoni, On the prime divisors of the sequence $w_{n+1} = 1 + w_1 \dots w_n$, \textit{J. Lond. Math. Soc.} \textbf{32 (2)} (1980), 1--11.

\bibitem{Odoni2}
R.\,W.\,K. Odoni, The Galois theory of iterates and composites of polynomials, \textit{Proc. Lond. Math. Soc.}  \textbf{51 (3)} (1985), 385--414.


\bibitem{ERF}
D.\,J.\,S. Robinson, and A. Russo and G. Vincenzi, On groups whose subgroups are closed in the profinite topology, \textit{J. Pure Appl. Algebra} \textbf{213 (4)} (2009), 421--429.

\bibitem{Scott}
P. Scott, Subgroups of Surface Groups are Almost Geometric, \textit{J London Math. Soc.} \textbf{s2-17 (3)} (1978), 555--565.

\bibitem{ShalevNewHorizons}
A. Shalev, Lie Methods in the Theory of pro-$p$ Groups, in: \textit{New Horizons in pro-p Groups}, Birkhäuser Boston, MA \textbf{1} (2000), 121--179.

\bibitem{Wilton}
H. Wilton, Hall’s theorem for limit groups, \textit{Geom. Funct. Anal.} \textbf{18 (1)} (2008), 271--303.

\end{thebibliography}

\end{document}